\documentclass{article}
\usepackage{amsmath,amsthm,amscd,amssymb}
\usepackage[mathscr]{eucal}
\usepackage{amssymb}
\usepackage{latexsym}

\theoremstyle{definition}
\newtheorem{Def}{Definition}[section]
\newtheorem{Thm}[Def]{Theorem}
\newtheorem{Prop}[Def]{Proposition}
\newtheorem{Rem}[Def]{Remark}
\newtheorem{Ex}[Def]{Example}
\newtheorem{Cor}[Def]{Corollary}

\newtheorem{Lem}[Def]{Lemma}

\numberwithin{equation}{section}

\title{On the kernel of the theta operator mod $p$}

\author{Siegfried B\"{o}cherer, Hirotaka Kodama and Shoyu Nagaoka}

\date{}
\begin{document}

\maketitle

\begin{abstract}
We construct many examples of level one Siegel modular forms
in the kernel of theta operators mod $p$ by using theta series
attached to positive definite quadratic forms.

\end{abstract}
\section{Introduction}
\label{intro}

Ramanujan's $\theta$ operator is a familiar topic in the theory of elliptic 
modular forms, defined by 
$$
f=\sum a(n)e^{2\pi i nz}\longmapsto \theta(f)=\frac{1}{2\pi i} 
f'=\sum n a(n)e^{2\pi i nz}
$$ 
For Siegel modular forms of degree $n$, the Fourier expansion runs over
positive semidefinite half-integral matrices of size $n$ and we can define 
several analogues of the Ramanujan $\theta$ operator: For $1\leq r\leq n$
we may introduce
\begin{align*}
F=\sum_T a(T) e^{2\pi i \text{tr}(TZ)}\longmapsto 
\Theta^{[r]}(F) &:= 
\frac{1}{(2\pi i)^r}\left(\frac{\partial}{\partial_{ij}}\right)^{[r]}F\\
                &=\sum_T T^{[r]}a(T)e^{2\pi i \text{tr}(TZ)},
\end{align*}
where, for a matrix $A$ of size $n$, we denote by $A^{[r]}$ the matrix
of all the determinants of its submatrices of size $r$ and 
$\partial_{ij}:=\frac{1}{2}(1+\delta_{ij})\frac{\partial}{\partial z_{ij}}$.

In general, $\Theta^{[r]}(F)$ is no longer a modular form, but it is a
modular form mod $p$ (even a $p$-adic modular form), vector-valued if $r<n$, 
see \cite{B-N3}.\\
Our aim in the present paper is to explore the existence and explicit 
construction of Siegel modular forms which are in the kernel of such 
$\Theta$-operators mod $p$. Obvious candidates for such modular forms are
 theta series $\vartheta_S^n$ 
$$
\vartheta^n_S(Z)=\sum_{X\in {\mathbb Z}^{(n,n)}}e^{2\pi i \text{tr}({}^tXSXZ)},
$$
attached to 
positive quadratic forms $S$ of rank $n$ and level being a positive power of $p$; 
we will also consider variants
of this involving a harmonic polynomial.  
Looking at the Fourier expansion, evidently such theta series
are in the kernel of $\Theta^{[n]}$ mod $p$. On the other hand, these theta series
are not of level one and one has to do a level change to level one.
This method only works for even degree $n$, because we otherwise
enter into the realm of modular forms of half-integral weight, but we shall
exhibit a somewhat weaker variant of our method also for the case of odd $n$. 
On the other hand, for even degree, our method provides plenty of 
examples for level one
forms $F$ (of weight in an arbitrary congruence class modulo $p-1$) which
satisfy $\Theta^{[j]}(F)\equiv 0 \pmod{p}$ and $\Theta^{[j-1]}(F)\not\equiv 0 \pmod{p}$.
Here $j$ is almost arbitrary, the only obstruction comes from the
arithmetic of quadratic forms, which puts some constraint on $(n,j,p)$.
\\
We have to make an important comment on what we mean by  
``explicit construction''
here: The kernel of $\Theta^{[j]}$ mod $p$ is a notion which depends only on 
modular forms mod $p$, therefore the weight of the constructed modular form is
only of interest mod $(p-1)$. On the other hand, one is also interested in
explicit small weights for which we can get modular  
forms in the kernel mod $p$. In this paper we address both versions of explicit
construction, we will call them ``weak construction'' and 
``strong construction'' respectively; in most cases our ``strong construction''
also gives the smallest possible weight, which is called ``filtration'' in the work of Serre
and Swinnerton-Dyer, see \cite{B-K-T} for details. 
In the final section we also show that some of the known examples of 
congruences
for  degree two Siegel modular forms can be explained by our methods. 
\\
Finally we remark that most of our results are formulated for odd primes only.
The reader interested in $p=2$ may adjust some of our results and methods
to $p=2$.

\section{Preliminaries}
\label{Prel}
\subsection{Siegel modular forms}
\label{Siegel}
For standard facts about Siegel modular forms we refer to \cite{A,F,Kl}. 
The group $Sp(n,\mathbb R)$ acts on the upper half space $\mathbb H_n$ in the usual way. 
For an integer $k$, a function $f:\mathbb H_n \longrightarrow \mathbb C$ and 
$M=\binom{AB}{CD}$ we define the slash operator by 

$$(f\mid_k M)(Z):=\mathrm{det}(CZ+D)^{-k}f((AZ+B)(CZ+D)^{-1}).$$
For a congruence subgroup $\Gamma $ of $Sp(n,\mathbb Z)$ and a character $\chi $ of $\Gamma$ 
we denote by $M_k(\Gamma ,\chi)$ the space of Siegel modular forms for $\Gamma $ of weight 
$k$ and character $\chi $ and $S_k(\Gamma ,\chi)$ the subspace consisting of
cusp forms.
If $\chi $ is trivial, we just omit it. We will mainly be concerned 
with congruence subgroups of type
$$
\Gamma _0^n(N):=\left\{M=\binom{AB}{CD}\mid C\equiv 0\bmod{N}\right\}
$$
and with groups arising from these by conjugation within $Sp(n,\mathbb Z)$. If $N=1$ we just write 
$\Gamma ^n$ instead of $\Gamma_0^n(1)$. The only characters of $\Gamma_0^n(N)$ occuring are 
those arising from Dirichlet characters mod $N$ in the usual way (i.e. $\chi (M)=\chi (\mathrm{det}D)$), 
the most important one will be the quadratic character
$$
\chi _p(*):=\left(\frac{(-1)^{\frac{p-1}{2}}p}{*}\right)
$$
for an odd prime $p$.
If $f$ is an element of $M_k(\Gamma)$ for an arbitrary $\Gamma$, then $f$ has a Fourier expansion
$$
f(Z)=\sum_Ta(T;f)e^{2\pi i\mathrm{tr}(TZ)}
$$
where $T$ runs over positive semidefinite rational symmetric matrices with bounded denominator. 
In particular, for $\Gamma =\Gamma _0^n(N)$, $T$ runs over positive semidefinite matrices in 
$$\Lambda _n:=\left\{T=(t_{ij})\in \mathrm{Sym}_n(\mathbb Q)\mid t_{ii},2t_{ij}\in \mathbb Z\right\}.$$ 

For a subring $R$ of $\mathbb C$ we denote by $M_k(\Gamma )(R)$ the set of $f\in M_k(\Gamma )$ 
whose Fourier coefficients lie in $R$.
\subsection{Traces}
\label{Traces}
We need the explicit form of the trace map
$$
\mathrm{Tr}:\begin{cases}M_k(\Gamma_0^n(p)) \longrightarrow M_k(\Gamma^n)\\
\ \ \ \ \ \ f \ \ \ \ \ \longmapsto \sum_\gamma f\mid _k\gamma \end{cases}
$$
where $\gamma $ runs over $\Gamma_0^n(p)\setminus \Gamma ^n$, see also \cite{B-F-S}. 
To obtain an explicit set of representatives for these cosets we start from a Bruhat decomposition 
over the finite field $\mathbb F_p:$
$$
Sp(n,\mathbb F_p)=\bigcup_{j=0}^nP(\mathbb F_p)\cdot \omega_j \cdot P(\mathbb F_p),
$$
where $P\subset Sp(n,\mathbb F_p)$ denotes the Siegel parabolic defined by $C=0$ and for $0\leq j\leq n$ we put
$$
\omega _j=\begin{pmatrix}0_j&0&-1_j&0\\0&1_{n-j}&0&0_{n-j}\\1_j&0&0_j&0\\0&0_{n-j}&0&1_{n-j}\end{pmatrix}.
$$
Using the Levi decomposition $P=MN$ with Levi factor
$$
M=M_n=\left\{m(A)=\begin{pmatrix}A&0\\0&{}^tA^{-1}\end{pmatrix}\mid A\in GL_n(\mathbb F_p)\right\}
$$
and unipotent radical
$$
N=\left\{n(B)=\begin{pmatrix}1_n&B\\0&1_n\end{pmatrix}\mid B\in M_n(\mathbb F_p) \ \ \text{symmetric}\right\}
$$
we easily see that
$$
\left\{\omega _j\cdot n(B_j)\cdot m(A)\mid B_j\in M_j(\mathbb F_p) \ \text{symmetric}, A\in P_{n,j}(\mathbb F_p)
\setminus GL_n(\mathbb F_p)\right\}
$$
is a complete set of representatives for the cosets $P(\mathbb F_p)\setminus P(\mathbb F_p)\cdot \omega _j\cdot P(\mathbb F_p)$. 
Here $M_j$ is naturally embedded into $M_n$ by $B_j\longmapsto \begin{pmatrix}B_j&0\\0&0_{n-j}\end{pmatrix}$ and 
$P_{n,j}$ is a standard parabolic subgroup of $GL_n$ defined by $0^{(n-j,j)}$ being the lower left corner of $g$. 
We tacitly identify the matrices above with corresponding representatives with entries in $\mathbb Z$ and obtain sets of 
representatives for $\Gamma _0^n(N)\setminus \Gamma^n$.\\
We analyse the contribution of fixed $j$ to the trace of a given $f\in M_k(\Gamma_0^n(p)):$\\ 
The function $f\mid _k\omega _j$ is itself a modular form for the group conjugate to $\Gamma_0^n(p)$ by $\omega_j$, 
it has a Fourier expansion
$$
f\mid _k\omega _j(Z)=\sum_{T=(t_{lm})}a_{j}(T;f)e^{2\pi i\mathrm{tr}(TZ)},
$$
where the $t_{lm}$ are integral or semi-integral except for the $t_{lm}$ in the upper left block of size $j$ in $T$, 
where $p$ may occur in the denominator. Then an elementary calculation using orthogonality of exponential sums shows that

$$\sum_{B_j}(f\mid_k\omega_j)\mid_k n(B_j)(Z)=p^{\frac{j(j+1)}{2}}\sum_{T\in \Lambda _n}a_{j}(T;f)e^{2\pi i\mathrm{tr}(TZ)}.$$ 
The result of the action of the matrices $m(A)$ is $:$ 
$$
\sum_{B_j,A}f\mid _k(\omega _j\cdot n(B_j)\cdot m(A))=p^{\frac{j(j+1)}{2}}
\sum_{T\in \Lambda _n}b_{j}(T;f)e^{2\pi i\mathrm{tr}(TZ)}
$$
with
$$
b_{j}(T;f)=\sum_{A}a_{j}\left(A^{-1}T\ ^tA^{-1};f\right).
$$
We can therefore write the contribution of a fixed $j$ to the trace as
$$
p^{\frac{j(j+1)}{2}}f\mid_k\omega_j \mid \widetilde{U} _j(p),
$$
where $\widetilde{U} _j(p)$ is an operator which maps a Fourier series to a new Fourier series, 
where the new coefficients are certain finite sums of the Fourier coefficients in the series we started from. 
Most of time the exact shape of this operator will not be important for us. 
At some point however we have to consider the actions of $n(B_j)$
and $m(A)$ separately and we split the operator 
accordingly into two pieces as
\begin{equation}
\label{Up}
\widetilde{U}_j(p)= \widetilde{U}^0_j(p) \circ D_j(p).
\end{equation}
We just mention the extreme cases$:$ For $j=0$ the operator $\widetilde{U}_0(p)$ is just the identity 
and $\widetilde{U} _n(p)$ is quite similar to the usual $U(p)$-operator$:$ 
$$
\widetilde{U} _n(p): \sum_{T\in \frac{1}{p}\Lambda _n}a(T)e^{2\pi i\mathrm{tr}(TZ)}\longmapsto 
\sum_{T\in \Lambda _n}a(T)e^{2\pi i\mathrm{tr}(TZ)}.
$$
Using this terminology we decompose the trace into $n+1$ pieces:

\begin{Prop}
For $f\in M_k(\Gamma_0^n(p))$
$$
\mathrm{Tr}(f)=\sum_{j=0}^n Y_j
$$
with
$$
Y_j:=p^{\frac{j(j+1)}{2}}(f\mid_k\omega_j)\mid \widetilde{U} _j(p).
$$
\end{Prop}

\begin{Rem}
It should be clear that this expression for the trace has an analogue for the more general case of 
taking the trace from $\Gamma _0^n(NR)$ to $\Gamma _0^n(N)$ if $N$ and $R$ are coprime and $R$ 
is squarefree (see e.g.\cite{B-F-S}). 
\end{Rem}
\subsection{Congruences}
\label{Congruences}
For a prime number $p$ we denote by $\nu _p$ the normalized additive valuation on $\mathbb Q$ (i.e. $\nu_p(p)=1$). 
For a Siegel modular form $f\in M_k(\Gamma )(\mathbb Q)$ with Fourier expansion 
$f(Z)=\sum_Ta(T;f)e^{2\pi i\mathrm{tr}(TZ)}$ 
we define
$$\nu _p(f):=\mathrm{min}\left\{\nu _p(a(T;f))\mid T\geq 0\right\}.$$
Note that this minimum is well defined because the Fourier coefficients of 
$f$ have bounded denominators.
We also remark that $\nu_p(f)$ makes sense not only for 
modular forms with rational Fourier coefficients but also 
for the general case $f\in M_k(\Gamma)({\mathbb C})$ by tacitly 
extending the valuation to the field generated by all Fourier coefficients. 
For two modular forms $f$ and $g$ we define 
$$
f\equiv g \pmod{p}:\Longleftrightarrow \nu _p(f-g)\geq 1+\nu _p(f).
$$
We finally remark that in this setting $\nu _p(f\mid _k\gamma )$ 
also makes sense for arbitrary $\gamma \in \Gamma^n$. 
In particular, for $f\in M_k(\Gamma_0^n(p),\chi )$ we may consider $\nu _p(f\mid _k \omega _j);$ 
the Fourier expansions of 
$f\mid _k \omega_j$ may be 
viewed as `` expansion of $f$ in the cusp $\omega _j$ ''. 
(Strictly speaking, we should consider the 
double coset $\Gamma_0^n(p)\cdot \omega _j\cdot P(\mathbb Z)$ 
as a cusp for $\Gamma_0^n(p);$ by abuse of 
language we will call the $\omega _j$ `` the cusps for $\Gamma _0^n(p)$'' ). 
For basic results concerning fields generated by Fourier coefficients
and boundedness of denominators we refer to \cite{Shi}.
We will however use these notions in the sequel only for theta series, 
where such 
properties are accessible in a much more elementary way.

\subsection{Lattices and theta series}
\label{Lattices}
For an even integral positive definite matrix $S$ of size $m=2k$ we define the degree $n$ theta series in the usual way:
$$
\vartheta _S^n(Z):=\sum_{X\in \mathbb Z^{(m,n)}}e^{\pi i\mathrm{tr}(S[X]Z)} \ \ \ \ (Z\in \mathbb H_n).
$$
where $S[X]=\ ^tXSX$. We will freely switch between the languages of matrices $S$ and corresponding lattices $L$ 
and we write sometimes $\vartheta^n(L)$ instead of $\vartheta ^n_S$. For the transformation properties of 
such theta series see e.g. \cite{A}.
 Following \cite{B-N2} a lattice $L$ will be called $p$-special, if it has an isometry of order $p$ with no fixed point 
in $L\setminus \{0\}$. The theta series of such a lattice automatically satisfies
$$
\vartheta ^n(L)\equiv 1 \pmod{p}.
$$
There are many such lattices, see \cite{B-N1}:
\begin{Prop}
Let $p$ be an odd prime, then there are $p$-special (positive definite, even) lattices of rank $p-1$, 
level $p$ and determinant $p^t$ for all $1\leq t\leq p-2$. 
\end{Prop}
Linear combinations of theta series for such lattices provide modular
forms with convenient congruence properties in {\it all} cusps,
see Theorem 2 and Theorem 2' in \cite{B-N1} (and also the erratum at the end
of this paper). 
\\[0.3cm] 
In the case $m=n$, we consider the series
$$
\vartheta_{S,\mathrm{det}}^{n}(Z):=\sum_{X\in \mathbb{Z}^{(n,n)}}\mathrm{det}X\cdot 
e^{\pi i\text{tr}(S[X]Z)},\ \ (Z\in \mathbb H_n).
$$
It is known that the series becomes a modular form of weight $1+\frac{n}{2}$ and vanishes identically 
if and only if there exists a matrix $U\in M_n(\mathbb Z)$ such that 
$$
S[U]=S,\ \ \ \ \ \ \ \ \mathrm{det}\,U=-1,
$$
(e.g. \ cf. \ \cite{F}).\\
As in \cite{Ma}, the theta series of this type allows us to construct  
cusp forms with accessible properties.

\subsection{Theta series in the cusps $\omega_j$} 
\label{ThetaSeries}

We need precise information about theta series in the 
cusps $\omega_j$, in particular
about the denominators, which occur in the Fourier expansions in the cusps.
To quote the results from \cite{BS,B-F-S},
it is more convenient to use the geometric notation here:
Let $(V,q)$ be a positive definite quadratic space over ${\mathbb Q}$
with attached bilinear form $B(x,y)=q(x+y)-q(x)-q(y)$ and 
let $L$ be an even lattice on $V$ (i.e. $q(L)\subseteq {\mathbb Z}$
of level $p$,  which means 
$q(L^{\sharp})\cdot {\mathbb Z}=\frac{1}{p} {\mathbb Z}$, where $L^{\sharp}$ 
denotes the dual of $L$.) After fixing a basis of $V$ we may identify
$V$ with ${\mathbb Q}^m$. The associated Gram matrix with respect to this basis
will be denoted by $S$; for ${\bf x}=(x_1,\dots, x_n)\in V^n$ we 
then define the 
$n\times n$ matrix $q({\bf x})$ by
$q({\bf x})_{ij}=\frac{1}{2} B(x_i,x_j)$. 
Let $P: {\mathbb C}^{(m,n)}\longrightarrow {\mathbb C}$
be a pluriharmonic polynomial with 
$P(XA)=\text{det}(A)^{\nu}P(X)$ for $A\in GL(n,{\mathbb C})$.
In this language, with
$$
\vartheta^n(P,L,Z):=\sum_{{\bf x}\in L^n} P({\bf x}) e^{2\pi i\text{tr}(q({\bf x})\cdot Z)}.
$$
and
$$
\vartheta^{(n-j,j)}(P,L,L^{\sharp},Z):=\sum_{{\bf x}\in L^{n-j}\times (L^{\sharp})^j}
P({\bf x})\, e^{ 2\pi i \text{tr}(q({\bf x})\cdot Z)}
$$
we get for all $j$
$$
\vartheta^n(P,L,Z)\mid_{\frac{m}{2}+\nu} \omega_j=
(\gamma_ps_p(V))^j \det(L)^{-\frac{j}{2}}\vartheta^{(n-j,j)}(P,L,L^{\sharp},Z).
$$ 
Here  $s_p(V)$ is the Hasse-Witt invariant at $p$ (as normalized in 
\cite{Sch}) and $\gamma_p$ depends only on 
$\det(L)\cdot ({\mathbb Q}_p^{\times})^2$, in particular it is one, if 
$\det(L)$ is a square.
There are obvious generalizations of this, if $L$ is of more general level.

\section{mod $\boldsymbol{p}$ kernel of theta operator}
\label{Kernel}
We are interested in constructing Siegel modular forms $F$ of level one, scalarvalued, such that
$F\not\equiv 0\bmod p$ and $\Theta^{[j]}(F)\equiv 0\bmod p$ for some $j>0$.	  
Thanks to the results on level changing in \cite{B-N3} it is sufficient to construct such modular forms for
groups $ \Gamma_0^n(p^N)$, possibly with nebentypus character $\chi_p$.
\subsection{Some generalities}
\label{generalities}
\begin{Rem} If $\Theta^{[r]}(F) \equiv 0 \pmod{p}$, then $\Theta^{[r']}(F)\equiv 0 \pmod{p}$ 
for all $r'\geq r$.
\end{Rem}
\begin{Rem} 
Suppose that 
$F=\sum a(T;F)e^{2\pi i\mathrm{tr}(TZ)}\in M_k(\Gamma^n)$ is mod $p$ singular of rank $l$, i.e.
$a(T;F)\equiv 0 \pmod{p}$ for all $T$ with $\text{rank}(T) >l$ and 
$a(T_0;F)\not\equiv 0 \pmod{p}$ for some $T_0$
of rank $l$, then $F$ automatically satisfies 
$\Theta^{[j]}(F)\equiv 0 \pmod{p}$ for all $j>l$.
Such modular forms were investigated in \cite{B-K}, they always satisfy 
$$
2k-l\equiv 0 \pmod{p-1}.
$$  
\end{Rem}
\subsection{Weak Constructions}
\label{Weak}
Let $S$ be an even integral positive definite matrix of size $n$, with $n$ even. We assume that
the rank of $S$ over the finite field $\mathbb{F}_p$ is $r<n$ and that the 
level of $S$ is a power of $p$, denoted by $p^l$. The determinant of $S$ can then be written as  $p^d$. 
Then 
$\vartheta^n_S\in M_{\frac{n}{2}}(\Gamma_0^n(p^l),\chi_p^d)$ satisfies
\begin{eqnarray*}
\vartheta_S^n & \not\equiv & 0 \pmod{p},\\ 
\Theta^{[j]}(\vartheta^n_S)  &  \equiv  & 0 \pmod{p} \qquad (j>r).
\end{eqnarray*}
If moreover the order of $\text{Aut}_{\mathbb Z}(S)$ is coprime to $p$, then
we have the stronger property
$$
\Theta^{[r]}(\vartheta_S^n)\not\equiv 0 \pmod{p},
$$
because the coefficient of this Fourier series at $S$ is just
$\sharp \text{Aut}_{\mathbb Z}(S) \cdot S^{[r]}$.\\
Now we assume in addition that $S$ has no automorphism of determinant $-1$ 
and that the order of $\text{Aut}_{\mathbb Z}(S)$ is coprime to $p$. Then 
$\vartheta_{S,\det}^n\in S_{\frac{n}{2}+1}(\Gamma^n_0(p^l),\chi_p^d)$
and it also satisfies the congruences above.
\\[0.3cm]
{\bf Conclusion:} {\it Under the assumption, that quadratic forms $S$ with the properties above exist, 
we get Siegel modular forms {\rm (}cusp forms respectively{\rm )} $F$ of level one and with weight
congruent to 
$\frac{n}{2}+d\cdot\frac{p-1}{2}$ 
$\pmod{(p-1)}$ {\rm (} $\frac{n}{2}+1+d\cdot \frac{p-1}{2}$ respectively{\rm ) }
satisfying the congruences above.}
\\
\\
Note that in general the actual weight of the level one form (using \cite{B-N3} or similar techniques)
will be much larger.
\begin{Rem}
The arithmetic of quadratic forms tells us, under which conditions on 
$(n,p,r,d)$ 
such forms $S$ exist or not.
We also note that an inspection of the Fourier expansion shows that a  
set of $h$ pairwise inequivalent quadratic forms $S_i$
with the properties above gives a set of $h$ pairwise linearly
independent modular forms mod $p$ in the kernel of $\Theta^{[n]}$,
provided that none of the $S_i$ has an integral automorphism of order $p$.
\end{Rem}
\begin{Rem}
There is a simpler construction, which covers more weights modulo $p-1$ and 
produces modular forms in the simultaneous kernel mod $p$
for {\it all} $\Theta$-operators: Suppose
that $f\in M_k(\Gamma_0^n(p^l))({\mathbb Z}_{(p)})$ satisfies
$\nu_p(f)=0$. Then $f(pz)$ is congruent mod $p$ to a modular 
form $F$
of level one which satisfies $\Theta^{[j]}(F)\equiv 0\bmod p$
for all $j\geq 1$. An analogous statement holds true if $f$ has nontrivial 
nebentypus $\chi_p$.
\end{Rem}

The remark above is quite usefull: 
We recall from \cite{F} that 
$$ 
\Theta^{[r]}(f\cdot g)= \Theta^{[r]}(f)\cdot g +\cdots \quad ,$$
where $\cdots$ is an integral linear combination of products of entries
of $\Theta^{[i]}f$ and of entries of $\Theta^{[j]}g$ ($i+j=r$).
Combining the conclusion above with the remark 3.3. and keeping in mind
that there always exist (integral, cuspidal) modular forms of level one and 
sufficiently large 
weight $k$, provided that $kn$ is even, we obtain elements in the kernel
of $\Theta^{[j]}$ for almost all weights, if $n$ is even:
\begin{Cor}
Under the same assumption on the existence of quadratic
form $S$ as above we obtain the existence of modular forms {\rm (}resp.\, cusp forms{\rm )}
$F$ of level one and weight congruent mod $p-1$ to
$$ \frac{n}{2}+d\cdot\frac{p-1}{2}+k\qquad
\left(resp.\; \frac{n}{2}+1+d\cdot\frac{p-1}{2}+k\right), $$
which satisfy the congruence above. Here $k$ is now arbitrary mod $p-1$.
\end{Cor}
\subsection{A variant, in particular for odd degree 
}

For $n$ arbitrary we choose an even integer $m$ with $m>n$
and a positive definite even integral quadratic form of size $m$
with $\text{rank}_{{\mathbb F}_p}(S)=m'<n$ and level($S$) a power of $p$.
Then $\vartheta^n_S$ is in the kernel of $ \Theta^{[n]}\bmod p$.
This construction gives weaker results than before, because
a lot of linear dependencies among the  theta series may arise here
(over ${\mathbb Q}$ or over ${\mathbb F}_p$).\footnote{This case appeared in a
discussion with S.Takemori.}
\section{Strong Constructions}
\begin{Thm}
\label{Theorem1}
For $n\equiv 0 \bmod 4$ and all primes $p\geq n+3$ there exists 
$F\in M_{\frac{n}{2}+p-1}(\Gamma^n)$, $F\not\equiv 0 \pmod{p}$
such that $ \Theta^{[n-1]}(F)\equiv \Theta^{[n]}(F)\equiv 0 \pmod{p}.$
If the prime $p$ satisfies  $p\equiv 1 \pmod{4}$, the condition $p\geq n$ 
is sufficient.
\end{Thm}
\begin{Thm}
\label{Theorem2}
For $n\equiv 2 \pmod{4}$ {\rm (}if $p\equiv 3 \pmod{4}${\rm )} or $4\mid n$ 
{\rm (}if $p\equiv 1 \pmod{4}$ {\rm )}
and all primes 
$p\geq 2n+3$ there exists $F\in M_{\frac{n}{2}+\frac{p-1}{2}}(\Gamma^n)$, $F\not\equiv 0 \pmod{p}$
such that $ \Theta^{[n]}(F)\equiv 0 \pmod{p}$.
\end{Thm}

\begin{Rem} A similar result, using Eisenstein series was shown by the 
third author \cite{Nag}.
Our method of proof provides examples beyond Eisenstein series.
\end{Rem} 
The proofs of these theorems 
are quite similar, they rely on results from \cite{B-N1}; we mainly have to
assure the existence of appropriate even integral quadratic forms of rank $n$.
In all cases, we have just to consider $n\bmod 8$, because we may add even
unimodular quadratic forms as orthogonal summands if necessary.  
\\
To prove Theorem \ref{Theorem1}, we first observe, that there exist even integral 
quadratic forms $S$
in dimension $n\equiv 0 \bmod 4$ with $\text{det}(S)=p^2$, level $p$ and
with rank $n-2$ over ${\mathbb F}_p$: 
For $n=4$ we may choose a lattice corresponding to a maximal order in the
definite rational quaternion algebra ramified only at $p$. For $n=8$ 
we recall that locally at $p$ the even unimodular lattice $E_8$ corresponds to
an orthogonal sum of hyperbolic planes. The requested quadratic form is then 
shown to exist by scaling one of the hyperbolic planes by the factor $p$.
 The theta series attached to such $S$
satisfies
$$
\nu_p(\vartheta^n_S\mid_{\frac{n}{2}} \omega_j)= -j
$$  
for all $j$. Theorem 4 from \cite{B-N1} 
asserts the existence of $F\in M_{\frac{n}{2}+p-1}(\Gamma^n)$
with $\vartheta^n_S\equiv F \pmod{p}$.

Concerning the second theorem, even integral quadratic forms $S$ of level 
$p$ and determinant $p$
are known to exist in the dimensions mentioned in the theorem:
For $p\equiv 3\bmod 4$ and $n=2$ we may chose the binary form $(x,y)\mapsto
x^2+xy+\frac{p+1}{4}y^2$. For $p\equiv 3\bmod 4$ and $n=6$ we consider
the orthogonal sum of a binary form of determinant $p$ and a 
quaternary form of determinant $p^2$. A maximal overlattice then 
has the requested property (see the next section for the notion of maximality).
For $p\equiv 1 \bmod 4$ and $n=4$ the existence is shown e.g. in 
\cite[p.350]{Pet}. For $n=8$ we may take the orthogonal sum of quaternary 
forms of determinant $p$ and $p^2$ and go to a 
maximal overlattice. For a local-global
proof of existence in the case $p\equiv 1 \bmod 4$ see \cite[lemme 19]{Wald}.\\
We may then construct $F$ from $\vartheta^n_S$ by
$$
F:= \text{Tr}(\vartheta^n_S\cdot {\mathcal E}),
$$
provided that we can find ${\mathcal E}\in M_{\frac{p-1}{2}}(\Gamma^n_0(p),\chi_p)({\mathbb Z})$ 
with ${\mathcal E}\equiv 1 \pmod{p}$ and
$$
\nu_p({\mathcal E}\mid\omega_j)\geq -\frac{j^2}{2}+1\qquad (j\geq 1).
$$
Then all the $Y_j$ in the trace above are congruent zero mod $p$ except for
$Y_0$. The existence of such $\mathcal E$ is asserted by (the corrected
version of) Theorem 2, \cite{B-N1} (see, Erratum, Theorem 2').

\begin{Rem} Unfortunately, the procedure above does not work
for the theta series $\vartheta_{S,\det}^n$; it breaks down 
because we get an additional factor $p$ in the denominators of the $Y_j$.
This is the main reason
for a different approach in the next section .
\end{Rem}

\section{Strong construction with harmonic polynomial}

\subsection{Motivation}
\label{Motivation}

We want to explain the numerical examples for degree 2 (see the last section)
by a refined version of our construction, in particular, we aim at  

\begin{Prop} 
\label{binary}
Let $p$ be a prime congruent 3 mod 4,
let $S$ be a binary quadratic 
form of discriminant $-p$ {\rm (}without improper automorphism {\rm )}.
Then there exists $F\in S_{2+3\cdot\frac{p-1}{2}}(\Gamma^2)$ such that
$$
F\equiv \vartheta^2_{S,\det} \pmod{p}.
$$
\end{Prop}
  
This will be proved in a more general  framework, using a variant of 
the calculus of traces.
Our method  strongly relies on the arithmetic of maximal lattices; 
we recall that a lattice $L$ in a quadratic space $(V,q)$ with 
$q(L)\subset {\mathbb Z}$ is called maximal, iff the only overlattice
$L'$
of $L$ with $q(L')\subset {\mathbb Z}$ is $L$ itself. For basic properties of maximal lattices we refer to \cite{B-Ne}.
In particular, even integral positive definite matrices $S$ with
squarefree determinant correspond to maximal lattices. 
The proposition from above is just a special case of the following general result

\begin{Thm}
\label{Theorem 3}
Let $n$ be an even positive integer and $p$ a prime with $p\equiv 3\bmod 4$ 
satisfying
$n+p-1\equiv 0 \bmod 4$. 
Let $S$ be an even positive definite quadratic form of rank $n$ 
and with $\det(S)=p$. Then there exists $F\in S_{\frac{n}{2}+1+3\frac{p-1}{2}}(\Gamma^n)$ such that
$$F\equiv \vartheta^n_{S,\det} \pmod{p}
$$ 
\end{Thm}

\begin{Rem} A similar result should hold for primes $p\equiv 1\bmod 4$, but our
method seems not to be applicable
in this case.

\end{Rem}
\subsection{The case $n+p-1\equiv 4\pmod{8}$, $p\equiv 3\bmod 4$}

We start from the following situation (with $n$ and $p$ as above).

$S$ is a lattice of level and determinant $p$ with even rank $n$.
The existence of such $S$ is guaranteed, because of
$p\equiv 3 \pmod{4}$ and $n\equiv 2 \pmod{4}$.

Furthermore, let $L$ be a $p$-special lattice  of rank $p-1$, level and determinant $p$, we may e.g.
take the root lattice $A_{p-1}$, see \cite{B-N2}.

We observe that $S\perp A_{p-1}$ is maximal (of determinant $p^2$)  provided that
$n+p-1\equiv 4 \pmod{8}$ (in this case there is no even unimodular lattice;
we also note that for $n+p-1\equiv 0 \pmod{8}$ there is no maximal lattice of
determinant $p^2$).
Under these assumptions we consider
$$
p\cdot \text{Tr}\left( \vartheta^n_{S,\det}\cdot \vartheta^n(L) 
\vartheta^n(L) \vartheta^n(L)\right)=\sum_{j=0}^n Y_j.
$$
We can view this as the trace of a theta series attached to the quadratic form
$$
Q:=S\perp L\perp L\perp L
$$
of rank $2k=n+3(p-1)$ and with determinant $p^4$. 
The transformation properties of such theta series give
\begin{align*}
& \vartheta^n(Q)\mid_k \omega_j\\
&= s_p(V) ^j p^{-2j} 
\sum_{X,R_1,R_2,R_3} 
\det((X^{(n-j)},
S^{-1}X^{(j)}))\text{exp}(2\pi i \text{tr}{\mathcal L}\cdot Z)
\end{align*}
where
\begin{align*}
{\mathcal L}=S[X^{(n-j)},S^{-1}X^{(j)}]&+
L[R_1^{(n-j)},L^{-1}R_1^{(j)}]\\
& +L[R_2^{(n-j)},L^{-1}R_2^{(j)}]+
L[R_3^{(n-j)},L^{-1}R_3^{(j)}].
\end{align*}

Furthermore,  $s_p(V)$ is the Witt-invariant of the quadratic space underlying the lattice $Q$ ,
and we decompose $X\in {\mathbb Z}^{(n,n)}$ and 
$R_i\in {\mathbb Z}^{(p-1,n)}$ into components with $n-j$ and $j$ columns.

Evidently we have $\nu_p(Y_0)=1$ and from the transformation formula above
we obtain
for $j\geq 1$
$$
\nu_p(Y_j) \geq 1+ \frac{j(j+1)}{2}-2j -1
$$

The last summand $-1$ comes from the harmonic polynomial.

For $j\geq 4$ this is strictly positive and only for $1\leq j\leq 3$ need further investigation.
Actually, thanks to the arithmetic of maximal lattices, stronger
integrality properties will be shown: All contributions $Y_j$ will be integral
and only $j=1$ and $j=2$ may possibly be nonzero mod $p$ 

We analyse the contributions of $\vartheta^n(Q)\mid \omega_j \widetilde{U}_j(p)$ 
separately for fixed $j\geq 1$:

First we recall the decomposition of the Hecke operator $\widetilde{U}_j(p)$:
$$
\widetilde{U}_j(p)=\widetilde{U}_j^0(p)\circ D_j(p)
$$
(cf. $\S$ \ref{Traces}, (\ref{Up})).

Furthermore we recall that $L$ is $p$-special, therefore we can write
$$
\vartheta^n(L)\mid_{p-1} \omega_j\sim p^{-\frac{j}{2}} (1+\vartheta_j^n(L)^0),
$$
where $\vartheta_j^n(L)^0$ is a Fourier series with 
$\vartheta_j^n(L)^0\equiv 0 \pmod{p}$. Therefore,
\begin{equation}
\left(\vartheta^n(L)\mid \omega_j\right)^3\sim p^{-\frac{3j}{2}}
(1+3\vartheta_j^n(L)^0+p^2\cdot W_j)
\label{three}\end{equation}
where $W_j$ is a Fourier series with integral Fourier coefficients
and $\sim$ means equality up to a $p$-adic unit as factor.

According to (\ref{three}),  
$\left(\vartheta^n_{S,\det}\cdot \vartheta^n(L)^3\right)\mid_k{ \omega_j}\mid U_j^0(p)$ 
decomposes into 3 pieces:

To analyse them we recall a crucial property of 
maximal lattices of rank $l$: If $M$ denotes a Gram matrix corresponding
to such a lattice, then\\
For any ${\mathfrak x}\in M^{-1}{\mathbb Z}^l$ with $M[{\mathfrak x}]\in {\mathbb Z}$  there exists 
${\mathfrak x}^0\in {\mathbb Z}^l$ 
with ${\mathfrak x}= M\cdot {\mathfrak x}^0$, i.e.
$$M[{\mathfrak x}]= M[{\mathfrak x}^0],$$
see e.g. \cite{B-Ne}.

The first piece is (up to a unit mod $p$) equal to
\begin{align*}
& p^{-2j} \vartheta_{S,\det}^n\mid \widetilde{U}_j^0(p)\\
& =p^{-2j}\sum_X\text{det}(X^{(n-j)}, S^{-1}X^{(j)}) 
\text{exp}(2\pi i \text{tr}(S[X^{(n-j)},S^{-1}X^{(j)}]Z))\mid \widetilde{U}_j^0(p).
\end{align*}
The maximality of $S$ implies that $X^{(j)}=S\cdot \widetilde{X}$ with $\widetilde{X}\in {\mathbb Z}^{(n,j)}$.
Therefore this part just equals
$$
p^{-2j}\vartheta^n_{S,\det}.
$$
The second part is
\begin{align*}
3p^{-2j}
\sum_{X, Y} 
\det &(X^{(n-j)}, S^{-1}X^{(j)}) \\
& \cdot \text{exp}(2\pi i\text{tr}( 
S[X^{(n-j)},S^{-1}X^{(j)}]Z+ L[Y^{(n-j)},L^{-1}Y^{(j)}]Z)).
\end{align*}
Here $Y\in {\mathbb Z}^{(p-1,n)}$ satisfies the additional condition $Y\ne 0$.
We use that the lattice $S\perp L$ is maximal and we obtain by the same 
reasoning as before (keeping in mind that $L$ is $p$-special) that this
contribution equals (up to a unit)
$$
p^{-2j} \vartheta^n_{S,\det} \cdot \sum_{Y\not=0} \text{exp}( 2\pi i \text{tr}(L[Y] Z))
= p^{-2j+1}W,
$$
where $W$ is a Fourier series with integral Fourier coefficients.\\
Finally the last contribution will have its Fourier coefficients
in $$p^{-2j+1}\cdot {\mathbb Z}$$
because of the factor $p^2$ in front of $W_j$; note that the harmonic 
polynomial may bring in an additional 
$p$ in the denominator.

So far we have ignored the contribution of the operator $D_j(p)$; actually, in the first two contributions, 
thanks to the maximality of $S$ and $S\perp L$ the
result after applying $\widetilde{U}_j^0(p)$ is already invariant under 
$GL(n,{\mathbb Z})\hookrightarrow Sp(n,{\mathbb Z})$
and only the number $d(j)$ of left cosets enters defining $D_j(p)$ really matters; 
this is a group index 
$$
d(j)=[GL(n,{\mathbb F}_p): P_{n,j}({\mathbb F}_p)]=
\prod_{i=1}^j\frac{p^{j+i}-1}{p^i-1}
$$ 
see e.g. \cite{Kr}, p.73, in particular 
it is congruent 1 mod $p$.

Now we collect everything into a precise formula for $Y_j$  ($1\leq j\leq n$):
$$
Y_j= p^{1+\frac{j(j+1)}{2}-2j} s_p(V)^{j}\times
\left\{d(j)\cdot \vartheta^n_{S,\det} + 3(*) +(**)\right\}
$$
where $(*)$ and $(**)$ are Fourier series with coefficients divisible by $p$.

The $p$-power in front of $Y_j$ is  (for $0\leq j\leq n$) 
$$
p^1,\quad p^{0},\quad p^{0}, \quad p^{1},\quad \cdots
$$
where the $\cdots$ are divisible by $p$. Therefore only $j=1$ and $j=2$ is relevant.

We summarize our calculation in the formula
\begin{equation}
p\cdot {\rm Tr}(\vartheta^n_{S,\det}\cdot\vartheta^n(L)^3)\equiv 
\left(s_p(V)+s_p(V)^2\right)\vartheta_{S,\det}^n\pmod{p}.
\label{Witt}
\end{equation}
 
Concerning $s_p(V)$, we remark that
the dimension of $V$ is $n+3(p-1)$; then $s_p(V)=1$ iff $n+3(p-1) 
\equiv 0 \pmod{8}$;
this is assured for $p\equiv 3 \pmod{4}$ under the assumption 
$n+p-1\equiv 4 \pmod{8}$.\\
This proves Theorem 3 for the case in question.

\subsection{Variant: The case $n+p-1\equiv 0 \pmod{8}$}
\label{variant}
In this case,
there are no maximal lattices of determinant $p^2$ of rank $n+p-1$; 
therefore we take an auxiliary prime $q$ different from $p$
and we consider $q\cdot S\perp L$.
This quadratic form is of determinant $q^n\cdot p^2$ and of rank $n+p-1$.
Its Witt invariant at $q$ can be computed as
$$s_q(q\cdot S\perp L)=s_q(S)\cdot \left(\frac{-1}{q}\right)^{\frac{n(n-1)}{2}}\cdot \left(\frac{p}{q}\right)^{n-1}= 
\left(\frac{-p}{q}\right)$$
We choose $q$ such that 
\begin{equation}
\left(\frac{-p }{ q}\right)=-1
\label{cong}
\end{equation}
holds.
In view of $s_{\infty}(q\cdot S\perp L)=1$ this implies that $s_p(q\cdot S\perp L)=-1$, in particular
both $q\cdot S$ and $q\cdot S\perp L$ are maximal at $p$.

Then we can use essentially the calculation above
to find a modular form $F$ of level $q$ and weight $\frac{n}{2}+1+3\cdot\frac{p-1}{2}$
such that $F\equiv \vartheta^n_{qS,\det}\pmod{p}$.
We have to mention one subtle point here:
The quadratic space $V'$ in question is now given by $qS\perp L\perp L\perp L$ with
Witt invariant
$$s_q(V')=-1 \quad, s_{\infty}(V')=-1$$
and therefore (by the product formula) $s_p(V')=1$. This is necessary for the
nonvanishing mod $p$ of the analogue of (\ref{Witt}).
The lemma below then assures the existence of a modular form $\tilde{F}$ of level one
and weight $\frac{n}{2}+1+3\cdot\frac{p-1}{2}$ such that 
$$\tilde{F}\equiv \vartheta^n_{S,\det}\pmod{p}.$$
\begin{Lem} 
\label{Lem1}
Let $p$ and $q$ be different primes with $\prod_{i=1}^n(1+q^i)$
coprime to $p$.  
Let $f\in M_k(\Gamma_0^n(p),\chi_p)({\mathbb Z}_{(p)})$ be given and 
assume that there exists 
$G\in M_l(\Gamma_0^n(q))({\mathbb Z}_{(p)})$ for some weight $l$
such that $$g\equiv G \pmod{p} ,$$
where the modular form $g$ of level $pq$ is given by $g(Z):=f(qZ)$.
Then there exists a level one form $F$ of weight l such that
$$f\equiv F \pmod{p}.$$ 
\end{Lem}
\begin{proof}
We choose $E\in M_{l-k}(\Gamma_0^n(p),\chi_p)$ such that $E\equiv 1 \pmod{p}$.
\\
Then
$$g\cdot E(qZ)- G= p\cdot H,$$
where $H\in M_l(\Gamma_0^n(pq))$ has $p$-integral coefficients.
We apply the operator $U(q)$ to obtain
$$f\cdot E=G\mid U(q) + p\cdot H\mid U(q)$$
Now we take the trace from $\Gamma_0^n(pq)$ to $\Gamma_0^n(p)$ on both sides.
On the left side it just means that we multiply the function by the index
$[\Gamma_0^n(p):\Gamma_0^n(pq)]=\prod_{i=1}^n(1+q^i)$,
which is coprime to $p$; the formula for the index can be found e.g. 
in \cite{Kl0}.
On the right side we observe that the trace does not affect the $p$-integrality,
in particular, the $H$-part plays no role mod $p$. This follows from the 
$q$-expansion principle, see e.g. \cite[Theorem 2]{Ich}.
Moreover, to compute the trace of $G\mid U(q)$, we observe that this trace
has level one, because $G\mid U(q)$ has level $q$.
\\
At the end, $f$ is congruent mod $p$ to a level one form of weight $l$. 
\end{proof}
\begin{Rem}: Note that it is possible to choose the prime $q$ in such a way 
that both the condition of the lemma and (\ref{cong}) are satisfied.
\end{Rem}
\begin{Rem}: It should be possible to prove the statement above 
concerning the $p$-integrality of a trace from
$\Gamma_0^n(pq)$ to $\Gamma_0^n(p)$ by a more elementary method 
for our special case (using properties of theta series),
avoiding the $q$-expansion principle.
\end{Rem}

\subsection{Some special properties}
We cannot expect that all elements in the kernel mod $p$ for the theta 
operator arise by theta series as in the sections above. In fact, our 
construction gives modular forms with amusing additional congruence properties:
\\
For a positive integer $d$ and a modular form $f\in M_k(\Gamma^n_0(N),\chi)$
with Fourier expansion $f(Z)=\sum_{T} a(T;f) e^{2\pi i tr(TZ)}$ we define
$$
a_d(f):=\sum_T \frac{1}{\epsilon^+(T)}a(T;f)
$$
where $T$ runs over representatives of the 
$SL(n,{\mathbb Z})$-equivalence classes of elements $T$ in $\varLambda_n$ with
$T>0$ and $\det(2T)=d$ and
$$
\epsilon^+(T):=\sharp\{U\in SL(n,\mathbb{Z})\,\mid\, T[U]=T\,\}.
$$
These numbers appear naturally in the Koecher-Maa{\ss} 
Dirichlet series attached to $f$, see \cite{Maass}: 
$$KM(f,s):=\sum_d a_d(f) d^{-s}$$  
For $f=\vartheta_{2S}^n$ and $f=\vartheta_{2S,\det}^n$ with half integral $S$ of size $n$
the Koecher-Maa{\ss} Dirichlet series are very special:
$$ KM(\vartheta^n_{2S},s)=\det(2S)^{-s}\sum_X \mid \det(X)\mid^{-2s}$$
$$KM(\vartheta^n_{2S,\det},s)=\det(2S)^{-2s} \sum_X \text{det}(X)\mid \det(X)\mid^{-2s}=0$$
Here $X$ runs over all non-degenerate integral matrices in ${\mathbb Z}^{(n,n)}$
modulo the action of $SL(n,{\mathbb Z})$.

\begin{Prop}
\label{average1}
Let $p$ be a prime, $n$ even, and $S_1,\dots S_h$ 
half integral positive definite matrices of size $n$ with $\det(2S_i)=p$
for all $i$; we consider the modular form
$$f:=\sum a_i \vartheta^n_{2S_i}$$
with $a_i\in {\mathbb Z}_{(p)}$.
Let $F$ be a modular form of level one with $F \equiv f \pmod{p}$.
Then 
$$a_d(F)\equiv 0 \pmod{p}$$
holds for all $d$ provided that $\sum_i a_i\equiv 0 \pmod{p}$.
\end{Prop}

\begin{Prop} 
\label{average2} 
Let $p$ be a prime, $n$ even, and $S$ a 
half integral positive definite matrix of size $n$ with $\det(2S)=p$; 
we consider the level one modular form $F$  with 
$F\equiv \vartheta^n_{2S,\det} \pmod{p}$. Then we have for all $d$
$$a_d(F)\equiv 0 \pmod{p}.$$
\end{Prop}
\section{Degree 2 case}

Explicit examples of degree $2$ modular forms in the kernel of the theta
operator mod $p$ were known before our construction:\\
 Let $\Delta _{12}$
denote the Ramanujan delta function and $[\Delta_{12}]$ the 
corresponding Klingen-type Eisenstein series of degree $2$. 
In \cite{B}, the first author showed that
$$
\Theta ([\Delta _{12}])\equiv 0 \pmod{23}.
$$
Then  Mizumoto \cite{Mi} showed that the Klingen-type 
Eisenstein series $[\Delta_{16}]$ attached to a weight $16$ elliptic cusp form 
satisfies the
congruence
$$
\varTheta ([\Delta_{16}])\equiv 0 \pmod{31}.
$$
The second and the third author and Kikuta found the congruence
$$
\varTheta (X_{35}) \equiv 0 \pmod{23},
$$
where $X_{35}$ is Igusa's cusp form of weight 35 (cf. \cite{K-K-N}). 
\\
They also predicted 
the existence of a form $X_{47}$ of odd weight $47$ satisfying
$$
\varTheta (X_{47}) \equiv 0 \pmod{31}.
$$
In fact, such modular form was constructed in \cite{K-N} by 
using Igusa's generators.\\
In subsection 6.3 we will show that these examples can all be explained by our 
constructions, using appropriate binary quadratic forms.
\subsection{Estimation of dimensions}
In this section, we consider the dimension of the kernel of theta
operator over $\mathbb{F}_p$.

Let $B_D$ be the set of positive-definite, integral binary quadratic forms
with discriminant $D$, i.e.,
$$
B_D:=\left\{\, S=\begin{pmatrix} a & b/2 \\ b/2 & c \end{pmatrix}\;\Big{|}\;
a,\,b,\,c\in\mathbb{Z},\,S>0,\,b^2-4ac=D\,\right\}.
$$
In the following, we mainly treat the case $D<0$.
We denote the class number by $h(D):=\sharp (B_D/SL(2,\mathbb{Z}))$.
\begin{Prop}
\label{dim}
Let $p$ be a prime number with $p \equiv 3 \pmod{4}$ and $p>3$.
Then
\vspace{1mm}
\\
{\rm (i)}\quad
${\rm dim}_{\mathbb{F}_p}\langle \widetilde{\vartheta}^2_{2S}\,\mid\,
             S\in B_{-p}\rangle_{\mathbb{F}_p} =\frac{h(-p)+1}{2}$,  
\vspace{1mm}                     
\\
{\rm (ii)}\quad
${\rm dim}_{\mathbb{F}_p}\langle \widetilde{\vartheta}^2_{2S,{\rm det}}\,\mid\,
             S\in B_{-p}\rangle_{\mathbb{F}_p} =\frac{h(-p)-1}{2}$,
\vspace{2mm}
\\
where $\widetilde{\vartheta}^2_{2S}$ {\rm (}resp. $\widetilde{\vartheta}^2_{2S,{\rm det}}${\rm )}
is the Fourier coefficientwise reduction modulo $p$ of $\vartheta^2_{2S}$
{\rm (}resp. $\vartheta^2_{2S,{\rm det}}${\rm )}.
\end{Prop}
\begin{Rem}
\label{Fp}
For a modular form $F$, $\widetilde{F}$ is considered as a formal power series over $\mathbb{F}_p$
via Fourier expansion (see \cite{Nag2}).
\end{Rem}
\begin{proof}
(i)\; Since the number of prime divisor dividing $D=-p$ is just one, 
$B_{-p}/SL(2,\mathbb{Z})$ has a unique ambiguous class
(cf. \cite{Z}, p.112, Korollar). 
We fix a representative $S_{p,0}$
of the ambiguous class and take a set of representatives of $B_{-p}/SL(2,\mathbb{Z})$
as
$$
\{\, S_{p,0},\;S_{p,i},\,\overline{S}_{p,i} \;\;(1\leq i\leq\tfrac{h(-p)-1}{2})\;\},
$$
where $S_{p,i}$ and $\overline{S}_{p,i}$ are $GL(2,\mathbb{Z})$-equivalent.
\\
We consider the Fourier expansion of $\vartheta^2_{2S_{p,i}}$:
$$
\vartheta^2_{2S_{p,i}}(Z)=\sum a(T;\vartheta^2_{2S_{p,i}})e^{2\pi i\text{tr}(TZ)}.
$$
Then we have
\begin{align*}
a(S_{p,j};\vartheta^2_{2S_{p,i}}) &= \sharp\{\,U\in GL(2,\mathbb{Z})\,\mid\, S_{p,i}[U]=S_{p,j}\,\}\\
                         &=
\begin{cases}
4 & \text{if $i=j=0$},\\
2 & \text{if $i=j>0$},\\
0 & \text{otherwise}.
\end{cases}
\end{align*}
This implies that
$$
\widetilde{\vartheta}^2_{2S_{p,i}}\qquad (1\leq i\leq\tfrac{h(-p)-1}{2})
$$
are linearly independent over $\mathbb{F}_p$.\\
(ii)\; It should be noted that
$\vartheta^2_{2S,{\rm det}}$ is non-trivial if and only if $S$ has no proper automorphisms.
A similar argument in (i) shows that
$$
\widetilde{\vartheta}^2_{2S_{p,i},{\rm det}}\qquad (1\leq i\leq\tfrac{h(-p)-1}{2})
$$
are linearly independent over $\mathbb{F}_p$. This proves (ii).
\end{proof}

For a prime number $p$, we define the following $\mathbb{F}_p$-vector spaces:
\begin{align*}
& \widetilde{V}_{k,p}:=\{\,\widetilde{f}\,|\, f\in M_k(\Gamma^2)(\mathbb{Z}_{(p)}),\,
\Theta (f) \equiv 0 \pmod{p}\,\},\\
& \widetilde{V}_{k,p}^{\text{cusp}}:=\{\,\widetilde{f}\,|\, f\in S_k(\Gamma^2)(\mathbb{Z}_{(p)}),\,
\Theta (f) \equiv 0 \pmod{p}\,\}.
\end{align*}
\begin{Rem}
As we noted in Remark \ref{Fp}, these spaces are considered as subspaces of certain
vector space consisting of formal power series over $\mathbb{F}_p$ (cf. \cite{Nag2}).
\end{Rem}
Considering Theorem \ref{Theorem 3}, we can get the following estimates:
\begin{Cor}
Under the same assumption in Proposition \ref{dim}, we have
\vspace{2mm}
\\
{\rm (i)}\quad ${\rm dim}_{\mathbb{F}_p}\widetilde{V}_{\frac{p+1}{2},p}\geq \frac{h(-p)+1}{2}$,
\vspace{2mm}
\\
{\rm (ii)}\quad ${\rm dim}_{\mathbb{F}_p}\widetilde{V}_{\frac{3p+1}{2},p}^{\rm cusp}\geq \frac{h(-p)-1}{2}$.
\end{Cor}
\begin{Ex}\; We consider the case $p=23$. In this case we have $h(-23)=3$. Under the
previous notation, we can take a set of representatives of $B_{-23}/SL(2,\mathbb{Z})$ as
$$
\label{S}
\left\{
S_{23,0}=\begin{pmatrix}
1 & 1/2 \\ 
1/2 & 6
\end{pmatrix},\;\;
S_{23,1}=\begin{pmatrix}
2 & 1/2 \\ 
1/2 & 3
\end{pmatrix},\;\;
\overline{S}_{23,1}=\begin{pmatrix}
2 & -1/2 \\ 
-1/2 & 3
\end{pmatrix}
\right\}.
$$
In this case, we have
$$
\widetilde{V}_{12,23}=\langle \widetilde{\vartheta}^2_{2S_{23,0}},
                              \widetilde{\vartheta}^2_{2S_{23,1}}
                      \rangle_{\mathbb{F}_{23}},
$$
namely we have
$$
\text{dim}_{\mathbb{F}_{23}}\widetilde{V}_{12,23}=2.
$$
We also obtain
$$
\widetilde{V}_{12,23}^{\rm cusp}=
\langle \widetilde{\vartheta}^2_{2S_{23,1},{\rm det}}
                      \rangle_{\mathbb{F}_{23}}.
$$
\end{Ex}
\subsection{Average of Fourier coefficients}
The congruence $\varTheta ([\Delta_{12}]) \equiv 0 \pmod{23}$ was proved by the first
author in \cite{B}. In this paper, he remarked that the ``average'' of the Fourier
coefficients is divisible by $23$ as below. This phenomenon can be explained by
the result obtained in Proposition \ref{average1}.
\\
\\
The congruences given below are proved by checking them numerically
for finitely many Fourier coefficients, using a ``Sturm bound'' from
\cite{C-C-K}.
\\[0.4cm]
\textbf{The case} $\boldsymbol{p=23}:$
\\
We use the notation above.
We have the congruence
$$
[\Delta_{12}] \equiv 12(\vartheta^2_{2S_{23,0}}-\vartheta^2_{2S_{23,1}})
\pmod{23}.
$$
By Proposition \ref{average1}, we obtain
$$
a_d([\Delta_{12}]) \equiv 0 \pmod{23}. 
$$
In the case of cusp forms, the following congruence holds
$$
X_{35} \equiv 12\vartheta^2_{2S_{23,1},{\rm det}} \pmod{23},\quad\text{and}\quad
a_d(X_{35}) \equiv 0 \pmod{23},
$$
(Proposition \ref{average2}).
\\
\\
\textbf{The case} $\boldsymbol{p=31}:$
\\
In this case, we have $h(-31)=3$.
$$
[\Delta_{16}] \equiv 16(\vartheta^2_{2S_{31,0}}-\vartheta^2_{2S_{31,1}})
\pmod{31}.
$$
This implies
$$
a_d([\Delta_{16}]) \equiv 0 \pmod{31}.
$$
In the case of cusp forms, we have
$$
X_{47} \equiv 16\vartheta^2_{2S_{31,1},{\rm det}} \pmod{31},\quad\text{and}\quad
a_d(X_{47}) \equiv 0 \pmod{31},
$$
where $X_{47}$ is the cusp form introduced before.
\\
\\
\textbf{The case} $\boldsymbol{p=47}:$
\\
This case is more interesting than the above two cases because $h(-47)=5$ and
$\text{dim}_{\mathbb{C}}S_{24}(\Gamma^1)=2$. We consider a couple of
Klingen Eisenstein series
$$
[E_4^3\Delta_{12}]\quad \text{and}\quad [\Delta_{12}^2],
$$
where $E_4$ is the normalized degree one Eisenstein series of weight $4$.
The following congruence relations hold:
\begin{align*}
& [E_4^3\Delta_{12}] \equiv
  24(\vartheta^2_{2S_{47,0}}-9\vartheta^2_{2S_{47,1}}+8\vartheta^2_{2S_{47,2}})
  \pmod{47},\\
& [\Delta_{12}^2] \equiv
  24(\vartheta^2_{2S_{47,1}}-\vartheta^2_{2S_{47,2}})
  \pmod{47}.
\end{align*}
Consequently
$$ 
a_d([E_4^3\Delta_{12}]) \equiv a_d([\Delta_{12}^2]) \equiv 0 \pmod{47}.
$$
\\
\begin{center}{\bf Erratum} \\
to our paper \cite{B-N1}
\end{center}

The formula in theorem 2 is incorrect, here is a correct version 
(Theorem 5 should then also be modified accordingly).\\
\\
{\bf Theorem 2}' :
{\it Assume that $p\geq 2n+3$.
Then there exists a modular form $h$ of weight $\frac{p-1}{2}$, level p
and nebentypus $\chi_p$ such that
$$   h\equiv 1 \pmod{p},$$
$$\nu_p(h\mid \omega_j)\geq -\frac{j^2}{2}+1\qquad(1\leq j\leq n).$$
}
\begin{proof}
We write down an explicit linear combination of theta series with the requested property:
Let $L_j$ denote a $p$-special lattice of rank $p-1$  and discriminant $2j+1$
with $0\leq j\leq n$.
We put $a_0=1$ and for $j\geq 1$:

$$a_j:=(-1)^j\cdot p^{\frac{j^2+j}{2}}$$

We define
$$h:= \sum_{j=0}^n a_j \vartheta^n(L_j)$$

Evidently, we have $h\equiv 1 \pmod{p}$.\\
We analyse the Fourier expansion  of $h $ in all cusps
$\omega_i$
with $1\leq i\leq n$:
$$
a_j\vartheta^n(L_j)\mid \omega_i= 
(-1)^{i} p^{-\frac{i(2j+1)}{2}}(-1)^j\cdot p^{\frac{j^2+j}{2}}\left(1+\cdots\right).
$$
We first check that
for $j<i-1$ and for $j>i$ the value of 
$\nu_p(a_j\vartheta^n(L_j)\mid \omega_i)$ is larger or equal to
$-\frac{i^2}{2}+1$:\\
The case $j=0$, $i\geq 2$ is clear because $-\frac{i}{2}\geq -\frac{i^2}{2}+1$.
\\
For $j\geq 1$  our claim is equivalent to
$$-i(2j+1)+j^2+j\geq -i^2+2$$
To consider the case $j<i-1$ we put
$i=j+1+t$ with $t\geq 1$.\\
We have to check that
$$-2(j+1+t)-(j+1+t)+(j+1+t)^2+(j+1+t)-2\geq 0$$
This expression equals $t^2+t-2$.\\
Now we consider the second case, i.e. $j=i+t$ with $t>0$:
We have to check that
$$-2(i+t)-i +(i+t)^2+i+t+i^2-2\geq 0$$
This expression equals $t^2+t-2$.\\[0.2cm]
It remains to consider the crucial cases $j=i-1$ and $j=i$. In both cases we seem to pick up a
denominator 
$p^{-\frac{i^2}{2}}$ 
when applying $\omega_i$ to $a_{i-1}\vartheta^n(L_{i-1})$ and to
$a_i\vartheta^n(L_i)$. 
The key point is to analyze the constant term of
\begin{equation}\left(a_{i-1}\vartheta^n(L_{i-1})+a_i\vartheta^n(L_i)\right)\mid \omega_i.
\label{constant}
\end{equation}
This is equal to
$$(-1)^{i+i-1}p^{\frac{-i(2i-1)+(i-1)^2+i-1}{2}}+p^{\frac{-i(2i+1)+i^2+i}{2}}=0$$
The nonconstant parts of the Fourier expansion of (\ref{constant})  
satisfy an additional congruence
mod $p$ because the $L_i$ are special lattices.
\end{proof}

Acknowledgement: We thank T.~Ichikawa, T.~Kikuta, R.~Schulze-Pillot and S.~Takemori
for helpful discussions related to this work.

S.~B\"{o}cherer\\
Kunzenhof 4B, 79117 Freiburg,\\
Germany\\
e-mail:boecherer@math.uni-mannheim.de
\\
\\
H.~Kodama\\
Academic Support Center, Kogakuin Univ.\\
Hachioji, Tokyo 192-0015 \\
Japan\\
e-mail:kt13511@ns.kogakuin.ac.jp\\
\newpage
\noindent
S.~Nagaoka\\
Dept. Mathematics, Kindai Univ.\\
Higashi-Osaka, Osaka 577-8502\\
Japan\\
e-mail:nagaoka@math.kindai.ac.jp

\end{document}